\documentclass{amsart}
\usepackage{amsmath,amsthm,amssymb}
\usepackage{lipsum} 
\usepackage[utf8]{inputenc}
\usepackage{esint}
\usepackage[all]{xy}
\usepackage{amsfonts} 
\usepackage{xcolor}
\usepackage{tikz}

\usepackage{color}
\usepackage[colorlinks=true]{hyperref}
\hypersetup{
    linkcolor=cyan,          % color of internal links
    citecolor=orange,        % color of links to bibliography
    filecolor=cyan,      % color of file links
    urlcolor=cyan
}

\usepackage{bbold}
\usepackage{enumerate}
\usepackage{calc}
\usepackage{bbold}
\usepackage[colorinlistoftodos,prependcaption,textsize=tiny]{todonotes}
\usepackage{todonotes}
\usepackage[shortcuts]{extdash}
\usepackage[style=alphabetic,maxnames=200,maxalphanames=6,giveninits]{biblatex}
\makeatletter
\allowdisplaybreaks
\numberwithin{equation}{section}
\theoremstyle{plain}
        \newtheorem{theorem}{Theorem}[section]

        \newtheorem{lemma}[theorem]{Lemma}

        \newtheorem{remark}[theorem]{Remark}

\newtheorem*{theorem*}{Theorem}
\newtheorem*{definition*}{Definition}
\newtheorem*{proposition*}{Proposition}

\newcommand{\R}{\mathbb{R}}

\newcommand{\cD}{\mathcal{D}}
\newcommand{\cH}{\mathcal{H}}
\newcommand{\Do}{\R^{2d}}

\def\aS{\mathsf S}
\def\aZ{\mathsf Z}
\def\az{\mathsf z}
\def\aF{\mathsf F}
\def\aG{\mathsf G}
\def\aJ{\mathsf J}
\def\av{\mathsf v}
\def\aw{\mathsf w}
\def\cS{\mathcal S}

\def\cL{\mathcal{L}}

\def\cA{\mathcal{A}}

\def\aL{\mathsf L}
\def\aM{\mathsf M}
\def\aaE{\mathsf E}

\def\cP{\mathbb{P}}

\newcommand\supp{\operatorname{Supp}}
\renewcommand{\d}{\partial} 

\newcommand{\dd}{{\,\rm d}}

\title{GENERIC formulation and small-angle limit for Kinetic wave equations}
\author[H.~Duong]{Manh Hong Duong}
\author[Z.~He]{Zihui He}

\address[H.~Duong]
{School of Mathematics, University of Birmingham, UK}
\email{h.duong@bham.ac.uk}
\address[Z.~He]
{Fakult\"at f\"ur Mathematik, Universit\"at Bielefeld, Postfach 100131, 33501 Bielefeld, Germany}
\email{zihui.he@uni-bielefeld.de}

\date{\today}

\begin{document}

\begin{abstract}
In this paper, we formulate the  three-wave and four-wave kinetic equations into the GENERIC framework and formally derive a small-angle limit for the four-wave equation. This limit is akin to the well-known grazing limit from the kinetic Boltzmann equation to the kinetic Landau equation. We also show the GENERIC structure of the limiting system.
\end{abstract}

\maketitle

\section{Introduction}
\subsection{GENERIC framework}\label{subsec:GENERIC}
In non-equilibrium thermodynamics, 
GENERIC (acronym for General Equation for Non-Equilibrium Reversible-Irreversible Coupling), which was first introduced in the context of complex fluids  \cite{GO97a,GO97b}, describes a class of thermodynamic evolution equations for  systems with both reversible and irreversible dynamics. It is formulated in terms of an energy functional, an entropy functional, and two geometric structures for the Hamiltonian and dissipative dynamics, such that it automatically fulfils the first and second laws of thermodynamics by ensuring conservation of energy and monotonic increase of entropy. Over the last three decades, the GENERIC framework has proved to be a versatile modelling tool for a variety of complex systems in physics and engineering, ranging from anisotropic inelastic solids to viscoplastic solids and thermoelastic dissipative materials. In recent years, there has been a considerable progress on establishing a mathematically rigorous foundation for the GENERIC framework \cite{mielke2011formulation,DPZ13,kraaij2020fluctuation,duong2023non,mielke2024deriving,EH25,DuongHe2025}, focusing on a rigorous derivation of GENERIC from Hamiltonian/stochastically interacting particle systems and on variational { formulations} for GENERIC systems. We refer the reader to the monographs \cite{Ott05,pavelka2018multiscale} for a detailed account, including many successful applications, of GENERIC. 

The main motivation for our current work is the GENERIC structure for the Boltzmann and Landau equations, which are two fundamental equations in the kinetic theory, and their connection via the so-called grazing limit. The GENERIC formulation for the kinetic Boltzmann equation has been shown formally in \cite{Ott97,grmela2018generic} and is proved rigorously in \cite{Erb23}
(for the spatially homogeneous model) and in \cite{EH25} (for the fuzzy model). The GENERIC structure for the kinetic Landau equation is only established recently in our works \cite{DuongHe2025} (formally for the classical model and rigorously for the fuzzy one). Furthermore, it has been known that the Landau equation can be obtained from the Boltzmann equation in the grazing limit, {  namely when the angle of collisions tends to zero \cite{villani1998new,godinho2013asymptotic,he2014asymptotic,carrillo2024landau}.}

Our aim is to show the GENERIC structure of the \textit{three-wave and four-wave kinetic equations}, which are central equations in the theory of wave turbulence. The wave kinetic equations are akin to the aforementioned Boltzmann equation in the statistical kinetic theory of particle systems. We also derive a similar small-angle limit, which is analogous to the grazing limit of the Boltzmann equation, for the four-wave equation.
\subsection{Wave turbulence theory and kinetic wave equations}
Wave turbulence theory refers to the statistical theory of weakly nonlinear dispersive waves. The central object in this theory is the wave kinetic equations (WKE). These equations describe the eﬀective behavior of a nonlinear wave system, where waves interact nonlinearly according to time-reversible dispersive or wave equations. 
This theory can be traced back to the works of
Peierls \cite{peierls1929kinetischen} on anharmonic crystals and of Hasselmann \cite{hasselmann1962non,hasselmann1963non}
on the energy spectrum of water waves. It was systematically
investigated by Zakharov and his collaborators \cite{zakharov2012kolmogorov}, notably on the Kolmogorov-Zakharov spectra, which is analogous to the famous Kolmogorov spectrum in hydrodynamic turbulence, which describes how energy cascades from large to small scales in a fluid.
%predict steady states of the corresponding microscopic system (possibly with forcing and dissipation at well-separated
%extreme scales), where the energy cascades at a constant flux through the (intermediate) frequency scales. 
Over the last decades, the theory has found important practical applications in a variety of areas, including oceanography and plasma physics, to
mention a few. More recently, in mathematics literature, there have been significant breakthroughs in the rigorous derivation of the wave kinetic equations from the nonlinear Schrödinger \cite{DH21,DH23,DH23b,buckmaster2021onset,hani2024inhomogeneous}. We also refer to \cite{zakharov2012kolmogorov,nazarenko2011wave} for a detailed exposition of the wave turbulence theory.

\subsection{GENERIC formulation of kinetic wave equations}
The wave kinetic equations share many similar properties with Boltzmann’s kinetic equation. Notably, they can both be formulated via collision operators and satisfy the conservation of energy and $H$-theorems---the corresponding (physically relevant) entropy functional associated to each equation increases over time. It is also worth noticing that the conservation of the energy and increase of the entropy functionals are also the key features of the GENERIC equation. 

Motivated by these striking similarities,  the important fact that the Boltzmann equation can be described by the GENERIC framework and the discussion at the end of Section \ref{subsec:GENERIC}, we ask the following natural questions: 
\begin{itemize}
    \item Can the WKE also be formulated into the GENERIC formalism?
    \item Can one derive a small-angle limit for the WKE, akin to the grazing limit from the Boltzmann equation to the Landau equation?
    \item Can the limiting system also be cast into the GENERIC framework?
\end{itemize}
The novelties of this paper are affirmative answers to these questions. {More precisely, we first formally formulate the three-wave and four-wave kinetic equations into the GENERIC framework. Then we rigorously prove the grazing limit of the collision operator in the four-wave equation, see Theorem \ref{thm:limit}. Finally, we also formally cast the the limiting system  into the GENERIC framework}. Furthermore, we also discuss {the similarities and differences with the Boltzmann/Landau equations for particle collisions} and possible research directions for future works.
\subsection{Organisation of the paper} The rest of the paper is planned as following: In Section \ref{sec:GENERIC}, we review the GENERIC framework, the three-wave and four-wave equations, and formulate the latter into the former. In Section \ref{sec:limit}, we show the small-angle limit for the four-wave
kinetic equations. Finally, further discussion and open problems for future research are presented in Section \ref{sec: conclusion}.

\subsection*{Acknowledgements}
M. H. D is funded by an EPSRC Standard Grant EP/Y008561/1. Z.~H. is funded by the Deutsche Forschungsgemeinschaft (DFG, German Research Foundation) – Project-ID 317210226 – SFB 1283. 
The authors are grateful to the anonymous referees for their helpful comments and suggestions.

\section{GENERIC formulation of kinetic wave equations}\label{sec:GENERIC}
The three- and four-wave kinetic equations are presented in Section \ref{sub-sec:WKE}. The general GENERIC framework is recalled in Section \ref{sub-sec:generic}. The GENERIC formulations of the three- and four-wave kinetic equations are provided in Sections \ref{sub-sec:3} and \ref{sub-sec:4}.
\subsection{Wave kinetic equations}
\label{sub-sec:WKE}
\subsubsection{Three-wave kinetic equation}
The three-wave kinetic equation governs the evolution of wave energy in systems where weakly nonlinear, resonant interactions between three waves lead to energy transfer, constrained by the conservation of both momentum and frequency. It is given by
\begin{equation}
\label{3WKE}
\partial_t f+\nabla_v\omega\cdot\nabla_x f=Q_3(f),   
\end{equation}
where 
\begin{equation*}
Q_3 (f)=\pi\int_{\mathbb{R}^{2d}} \Big[R_{v,v_1,v_2}(f)-R_{v_1,v,v_2}(f)-R_{v_2,v,v_1}(f)\Big]dv_1dv_2,    
\end{equation*}
with
\begin{equation}
\label{R3}
 R_{v,v_1,v_2}(f):=|V_{v,v_1,v_2}|^2\delta_d(v-v_1-v_2)\delta_1(\omega-\omega_1-\omega_2)(f_1f_2-ff_1-ff_2),   
\end{equation}
with the short-hand notations: $f=f(t,x,v), \omega=\omega(v)$ and $f_j=f(t,x,v_j)$ for $j\in\{1,2\}$, $\omega_j=\omega(v_j)$. The quantity $\omega(v)$ denotes the dispersion relation of the waves. For example, the Schrödinger dispersion relation is given by $\omega(v)=|v|^2$, and the Bogoliubov dispersion relation is given by $\omega(v)=\sqrt{C_1|v|^2+C_2|v|^4}$, $C_1,C_2>0$. Other examples of physically relevant dispersion relations can be found in \cite{germain2020optimal}. In \eqref{R3}, $\delta_d$ and $\delta_1$ denote the delta-distribution of dimensions $d$ and $1$. These Dirac measures capture the three-wave interaction $v\leftrightarrow v_1+v_2$, satisfying the following relations:
\[
v=v_1+v_2,\quad \omega=\omega_1+\omega_2.
\]
The exact form of the collision kernel $V_{v,v_1,v_2}\ge0$ depends on the type of waves under consideration, { for example, in acoustic wave case, $V_{v,v_1,v_2}\propto\sqrt{|v||v_1||v_2|}$.}   The collision operator $Q_3$ in \eqref{3WKE} takes into account all three-wave interactions.
\subsubsection{Four-wave kinetic equation}
The four-wave kinetic equation, in contrast to the three-wave equation, describes energy transfer via resonant interactions among four waves (wave quartets), and it applies in systems where three-wave interactions are forbidden (e.g., by symmetry or dispersion constraints), such as in deep-water gravity waves or nonlinear optics. The four-wave kinetic equation is given by
\begin{equation}
\label{4WKE}
\partial_t f+\nabla_v\omega\cdot\nabla_x f=Q(f),   
\end{equation}
where 
\begin{align*}
    Q(f)=&4\pi \int_{\R^{3d}}|V|^2\delta_d(v+v_2-v_1-v_3)\delta_1(\omega+\omega_2-\omega_1-\omega_3)\\
    &(f^{-1}+f_2^{-1}-f_1^{-1}-f_3^{-1})ff_1f_2f_3\dd v_1\dd v_2\dd v_3,
\end{align*}
with the short-hand notations: $f=f(t,x,v), \omega=\omega(v)$ and $f_j=f(t,x,v_j)$ for $j\in\{1,2,3\}$, $\omega_j=\omega(v_j)$. The quantity $\omega(v)$ again denotes the dispersion relation of the waves. The equation describes, under the assumption of weak nonlinearities, the spectral energy transferred on the resonant manifold, which is a set of wave vectors satisfying 
\[
v+v_2=v_1+v_3\quad\text{and}\quad \omega+\omega_2=\omega_1+\omega_3.
\]
The exact form of the collision kernel $V$ depends on the type of waves under consideration, { for example, in the nonlinear Schrödinger case $V=1$.}  In this paper, we write $V_{v123}=V(v,v_1,v_2,v_3)$, and assume the symmetric property that   
\begin{equation*}
    \label{sys:V}
    V_{v123}=V_{1v32}=V_{23v1}.
\end{equation*}
Under this symmetric assumption, the collision operator $Q$ in \eqref{4WKE} takes into account all other $2\leftrightarrow 2$-wave interactions. We refer to \cite{zakharov2012kolmogorov,nazarenko2011wave} for more information on wave kinetic equations and wave turbulence theory.

\subsection{GENERIC formulation}
\label{sub-sec:generic}

A GENERIC system describing the evolution of an unknown $\mathsf z$ in a state space $\mathsf Z$ is given by
{ \begin{equation}
\label{eq:generic}
\partial_t \az = \underbrace{\aL (\az) \dd\aaE(\az) }_{\text{reversible dynamics}}+\underbrace{\aM (\az) \dd\aS(\az) }_{\text{irreversible dynamics}}.
\end{equation}}
The key feature of \eqref{eq:generic} is the splitting of the evolution in to reversible and irreversible dynamics using two operators $\aL$ and $\aM$ and two functionals $\aaE$ and $\aS$. More specifically:
\begin{itemize}
\item The functionals $\aaE,\aS:\aZ\to \R$ are interpreted as, energy and entropy functionals  respectively; and $\dd\aaE$, $\dd\aS$ are their differentials.
\item $\aL(\az)$, {  for every $\az\in \aZ$, }  is an antisymmetric operator mapping cotangent to tangent vectors and satisfies the Jacobi identity
\begin{equation*}\label{eq:jacobi}
\{\{\aG_1,\aG_2\}_{\aL},\aG_3\}_{\aL}+\{\{\aG_2,\aG_3\}_{\aL},\aG_1\}_{\aL}+\{\{\aG_3,\aG_1\}_{\aL},\aG_2\}_{\aL}=0,
\end{equation*}
for all functions $\aG_i:\aZ\to \R$, $i=1,2,3$, where the Poisson bracket $\{\cdot,\cdot\}_{\aL}$ is defined as follows:
{ $$\{\aF(\az),\aG(\az)\}_{\aL}=\dd\aF(\az)\cdot\aL (\az)\dd\aG(\az).$$}
Here, $\cdot$ signifies the duality pairing of the cotangent and tangent spaces of $\aZ$. The antisymmetry of $\aL$ means that $\av\cdot \aL\aw= -\aw^*\cdot \aL\av^*$, where $\av^*,\aw^*$ refer to the dual vectors of $\av,\aw$;
\item $\aM(\az)$, {  for every $\az\in \aZ$, } is a symmetric and positive semi-definite operator mapping cotangent to tangent vectors in the sense that
$$\av\cdot \aM\aw=\aw^*\cdot \aM\av^*,\quad \av^*\cdot \aM\av\geq 0\quad\text{for any }\av,\,\aw;$$
\item The following degeneracy (non-interaction) conditions are satisfied:
\begin{equation*}
\label{eq:degen}
\aL(\az) \dd \aS(\az)=0\quad\text{and}\quad \aM(\az) \dd\aaE(\az) = 0\quad\text{for all }\az.
\end{equation*}
\end{itemize}
Note that pure Hamiltonian systems and pure (dissipative) gradient flow systems are special cases of GENERIC corresponding to $\aM\equiv 0$ and $\aL\equiv 0$, respectively. 

The conditions satisfied by the building blocks $\{\aaE, \aS, \aL,\aM\}$ ensure that in any solution to \eqref{eq:generic}, the energy $\aaE$ is conserved and the entropy $\aS$ is non-decreasing. In fact, 
\begin{align*}
\frac{\dd}{\dd t} \aaE(\mathsf z_t) = \partial_t z\cdot \dd\aaE=(\aL \dd\aaE+\aM \dd\aS)\cdot \dd\aaE=\underbrace{\aL \dd\aaE\cdot \dd \aaE}_{=0}+\underbrace{\aM \dd\aS \cdot \dd \aaE}_{=0}=0,
\end{align*}
{where the first cancellation, $\aL \dd\aaE\cdot \dd \aaE=0$, follows from the anti-symmetry of $\aL$, while the second one, $\aM \dd\aS \cdot \dd \aaE=0$, follows from the symmetry of $\aM$ and the second degeneracy condition}. By a similar computations, {  the positivity of $\aM$ ensures that }
$$
\frac{\dd}{\dd t} \aS(\az_t) = \dd \aS\cdot \aM \dd\aS \ge0.$$
Thus, the first and second laws of thermodynamics are automatically justified for GENERIC systems. We refer to \cite{Ott05,pavelka2018multiscale} for more information on the GENERIC framework.

We now bridge the two previous subsections together by formulating the three-wave and four-wave kinetic equations into the GENERIC framework, starting with the three-wave equation.
\subsection{GENERIC formulation of the three-wave kinetic equation}\label{sub-sec:3}
We define a discrete gradient operator $\overline\nabla_3$ for any $\phi=\phi(x,v)$ by
\begin{align}
\label{**}
\overline\nabla_3\phi=\overline\nabla_3\phi(x,v,v_1,v_2)=\phi_1+\phi_2-\phi, 
\end{align}
where we denote $\phi=\phi(x,v)$ and $\phi_i=\phi_i(x,v_i)$, $i\in\{1,2\}$. Let $G=G(x,v,v_1,v_2)$. We have the following integration by parts formula 
\begin{align*}
    \int_{\R^{4d}}G\cdot \overline \nabla_3 \phi\dd v\dd x\dd v_1\dd v_2=-\int_{\R^{2d}}\overline\nabla_3\cdot G \phi\dd v\dd x,
\end{align*}
where $\overline\nabla_3\cdot G$, which is a discrete divergence operator associated to $\overline\nabla_3$, is given by 
\begin{align}
\label{*}
\overline\nabla_3\cdot G(x,v)=\int_{\R^{2d}}G(x,v,v_1,v_2)-G(x,v_1,v,v_2)-G(x,v_2,v_1,v)\dd v_1\dd v_2.    
\end{align}
Using this discrete divergence operator, the collision operator in the three-wave kinetic equation \eqref{3WKE}, $Q_3(f)$, can be written as 
\begin{align*}
    Q_3(f)=-\overline\nabla_3\cdot \big(\delta_0\hat \delta_0|V_{v,v_1,v_2}|^2ff_1f_2\overline\nabla_3(f^{-1})\big),
\end{align*}
where we use the notations $\delta_0=\delta_d(v_1+v_2-v)$ and $\hat\delta_0=\delta_1(\omega_1+\omega_2-\omega)$.  

We have the following weak formulation for the wave kinetic equation \eqref{3WKE}
\begin{align*}
   &\int_{\R^{2d}}\phi_0f_0\dd v\dd x- \int_0^T\int_{\R^{2d}}(\d_t\phi +\nabla_v\omega\cdot \nabla_x\phi)f\dd v\dd x\dd t\\
   &=-\int_0^T\int_{\R^{2d}}\delta_0\hat \delta_0|V_{v,v_1,v_2}|^2ff_1f_2\overline\nabla_3\phi\cdot \overline\nabla_3 f^{-1}\dd v\dd x\dd t\quad\forall \phi\in C^\infty_c([0,T)\times\R^{2d}).
\end{align*}
Since $\overline\nabla_3(v,\omega)=0$, at least formally, we have the following momentum and energy conservation laws
\begin{align*}
    \int_{\R^{2d}}(v,\omega)f_t\dd v\dd x=\int_{\R^{2d}}(v,\omega)f_0\dd v\dd x\quad \text{for all $t\in[0,T]$}.
\end{align*}
We define the entropy
\begin{equation}
\label{eq: log entropy}
\cH(f)=\int_{\R^{2d}}\log f\dd v\dd x.    
\end{equation}
We take $\phi=\log f$ in the weak formulation to derive the following $\cH$-theorem
\begin{align}
\label{H}
    \frac{d}{dt}\cH(f)=\mathcal{D}_3(f),
\end{align}
where the entropy dissipation is defined as
\begin{align*}
    \mathcal{D}_3(f)=\int_{\R^{4d}}\delta_0\hat \delta_0|V_{v,v_1,v_2}|^2ff_1f_2|\overline \nabla_3 f^{-1}|^2\dd v\dd x\dd v_1\dd v_2\ge0.
\end{align*}

We consider the space $\aZ$ to be the space that consists of all Schwartz functions $f\in \cS(\R^{2d})$ such that $f\cL\in\cP(\R^{2d})$, { where $\cL$ denotes the Lebesgue measure on $\R^{2d}$, and $\cP(\R^{2d})$ denotes the space of Borel probability
measures on $\R^{2d}$.} The space is endowed with the $L^2$-inner product $\langle f,g\rangle=\int_{\R^{2d}}fg\dd v\dd x$. We formally consider the tangent and co-tangent space at $f\in\aZ$ as $T_f\aZ=\{\xi\in L^2(\R^{2d})\mid \int_{\R^{2d}}\xi\dd v\dd x=0\}$ and $T_f^*\aZ=L^2(\R^{2d})$.

In the case of the wave kinetic equation \eqref{4WKE}, we define the energy and entropy functional by 
\begin{equation*}
    \label{ES}
\begin{aligned}
&\aaE(f)=\int_{\R^{2d}} \omega f\dd v\dd x\quad\text{and}\quad \aS(f)=\cH(f).
\end{aligned}
\end{equation*}
We define the operators $\aL$ and $\aM$ at $f\in\aZ$ by
\begin{equation*}
    \label{ML}
\begin{aligned}
\aM(f)g&=-\pi\overline{\nabla}_3\cdot\Big(\delta_0\hat \delta_0|V_{v,v_1,v_2}|^2ff_1f_2\overline{\nabla}_3 g\Big)\quad\text{and}\\
\aL(f)g&=-\nabla\cdot(f \aJ\nabla g),\quad \aJ=\begin{pmatrix}
    0& \mathsf{id}_d\\
    -\mathsf{id}_d&0
\end{pmatrix}
\end{aligned}
\end{equation*}
for all $g\in \aZ$, where $\nabla =(\nabla_x,\nabla_v)^T$ denotes the usual gradient operator on the phase space.

We will check directly that the three-wave equation can be cast into the GENERIC framework \eqref{eq:generic}
with these building blocks. We first verify that they satisfy the GENERIC conditions. Indeed, { $\aL$ satisfies the Jacobi identity and the anti-symmetry of $\aJ$ implies that $\aL$ is also anti-symmetric  (see \cite{DPZ13})}
 \begin{align*}
&\langle \aL(f)g,g\rangle=\int_{\R^{2d}} f \nabla g \aJ\nabla g=0
\end{align*}
for all $g\in\aZ$.
 The operator $\aM$ is symmetric, positive semi-definite since, using the integration by parts formula, we have
\begin{align*}
\langle \aM(f)g,h\rangle&=-\pi\int_{\R^{2d}}\overline{\nabla}_3\cdot\Big(\delta_0\hat \delta_0|V_{v,v_1,v_2}|^2ff_1f_2 \overline{\nabla}_3 g\Big) h\\
&=\pi\int_{\R^{4d}}\delta_0\hat \delta_0|V_{v,v_1,v_2}|^2ff_1f_2 \overline{\nabla}_3 g\cdot \overline{\nabla}_3 h=\langle \aM(f)h,g\rangle
\\\text{and}\quad   \langle \aM(f)g,g\rangle&=\pi\int_{\R^{4d}} \delta_0\hat \delta_0|V_{v,v_1,v_2}|^2ff_1f_2 |\overline{\nabla}_3g|^2\ge 0
\end{align*}
for all $g,h\in\aZ$.

\medskip

We indeed recover the wave kinetic equation \eqref{4WKE} from the GENERIC equation, since
\begin{itemize}
    \item $\dd\aaE(f)=\omega(v)$ and 
    \begin{align*}
        \aL(f)\dd\aaE(f)&=-\frac12\nabla\cdot \big(f \aJ\nabla \omega\big)=-\nabla_v\omega\cdot\nabla_x f,
    \end{align*}
    \item $\dd\aS(f)=f^{-1}$ and
    \begin{align*}
        \aM(f)\dd\aS(f)&=-\pi\overline\nabla_3\cdot\Big(\delta_0\hat \delta_0|V_{v,v_1,v_2}|^2ff_1f_2\overline\nabla_3 f^{-1}\Big).
    \end{align*}
\end{itemize}
It remains to show the non-interaction conditions. We have, by the anti-symmetry of $\aJ$ and the definition of $\overline\nabla_3$:
\begin{align*}
\aL(f)\dd\aS(f)&=-\nabla\cdot\big(f\aJ\nabla f^{-1}\big) =-\nabla\cdot(-\nabla_v\log f,\nabla_x\log f)^T=0\\
\text{and}\quad\aM(f)\dd \aaE(f)&=-\pi\overline\nabla_3\cdot\Big(\delta_0\hat \delta_0|V_{v,v_1,v_2}|^2ff_1f_2 \overline\nabla_3 \omega\Big)=0.
\end{align*}

\subsection{GENERIC formulation of the four-wave kinetic equation}\label{sub-sec:4}

We first show the GENERIC formulation in the case of $\omega(v)=|v|^2$. In this case, the four-wave kinetic equation \eqref{4WKE} becomes
\begin{align}
\label{WKE}
    (\d_t+v\cdot\nabla_x)f=Q(f).
\end{align}
For $v,v_*\in\R^d$, the conservation of momentum and energy 
\begin{equation}
    \label{cl}
    v+v_*=v_*'+v',\quad |v|^2+|v_*|^2=|v_*'|^2+|v'|^2
\end{equation}
imply that, similar to Boltzmann equations,  $v_*',v'$ can be parametrised as follows~\cite{ampatzoglou2025inhomogeneous,ampatzoglou2024scattering}
\begin{align*}
v'=\frac{v+v_*}{2}+\frac{|v-v_*|}{2}\sigma\quad\text{and}\quad v_*'= \frac{v+v_*}{2}-\frac{|v-v_*|}{2}\sigma \quad \sigma\in S^{d-1}.
\end{align*}
We write 
\begin{align*}
    f_*=f(x,v_*),\quad f'=f(x,v')\quad\text{and}\quad 
f_*'=f(x,v_*').
\end{align*}
Then $Q(f)$ can be written as 
\begin{align*}
    Q(f)=4\pi\int_{\R^d\times  S^{d-1}}|V|^2ff_* f'f'_*(f^{-1}+f_*^{-1}-(f')^{-1}-(f'_*)^{-1})\dd\sigma \dd v_*. 
\end{align*}

We define $\overline\nabla$ for any $\phi=\phi(x,v)$ by
\begin{align}
\label{nabla-4} \overline\nabla\phi=\phi'+\phi'_*-\phi_*-\phi, 
\end{align}
where as above we denote $\phi_*=\phi(x,v_*),~\phi'=\phi(x,v')$ and $ 
\phi_*'=\phi(x,v_*')$. Let $G=G(x,v,v_*,v',v_*')$. We have the following  integration by parts formula
\begin{equation*}
\label{IP}
    \int_{\R^{2d}\times S^{d-1}}G\cdot \overline \nabla \phi \dd\sigma\dd v_*\dd v\dd x=-\int_{\R^{2d}}(\overline\nabla\cdot G)\phi\dd v \dd x,
\end{equation*}
where the divergence $\overline\nabla\cdot G$ is given by
\begin{align*}
\overline\nabla\cdot G(x,v)&=\int_{\R^{2d}\times S^{d-1}}\big(G(x,v,v_*,v',v_*')+G(x,v_*,v,v_*',v')\\
&-G(x,v',v_*',v,v_*)-G(x,v_*',v',v_*,v)\big)\dd \sigma\dd v_*.
\end{align*}
We assume that $|V|^2$ is invariant under the transformations between $v,v_*,v'$ and $v_*'$, then $Q(f)$ can be written as 
\begin{equation*}
\label{Q:bar}
    Q(f)=-\pi\overline\nabla\cdot \big(|V|^2ff_* f'f'_*\overline\nabla (f^{-1})\big). 
\end{equation*}

We have the following weak formulation for the 4-wave kinetic equation \eqref{WKE}
\begin{align*}
   & \int_{\R^{2d}}\phi_0f_0\dd v\dd x-\int_0^T\int_{\R^{2d}}(\d_t\phi +v\cdot\nabla_x\phi)f\dd v\dd x\dd t\\
   &=-\pi\int_0^T\int_{\R^{2d}}|V|^2ff_* f'f'_*\overline\nabla\phi\cdot \overline\nabla (f^{-1})\dd v\dd x\dd t\quad\forall \phi\in C^\infty_c([0,T)\times\R^{2d}).
\end{align*}
Since $\overline\nabla(1,v,|v|^2)=0$, at least formally, we have the following mass, momentum and energy conservation laws
\begin{align*}
    \int_{\R^{2d}}(1,v,|v|^2)f_t\dd v\dd x=\int_{\R^{2d}}(1,v,|v|^2)f_0\dd v\dd x\quad \text{for all $t\in[0,T]$}.
\end{align*}
Recalling the definition of the entropy $\cH(f)$ in \eqref{eq: log entropy}, similarly to the three-wave case \eqref{H}, we have
\begin{align*}
    \frac{d}{dt}\cH(f)=\mathcal{D}(f),
\end{align*}
where the entropy dissipation is defined as
\begin{align*}
    \mathcal{D}(f)=\pi\int_{\R^{2d}\times S^{d-1}}|V|^2ff'f_* f'_*|\overline \nabla f^{-1}|^2\dd\sigma\dd v_*\dd v\ge0.
\end{align*}

We closely follow the three-wave case in Section \ref{sub-sec:3} to define $(\aaE,\aS,\aM,\aL)$ with appropriate adjustment for the definition $\overline\nabla$ in \eqref{nabla-4} 
\begin{equation}
\label{block-4}
\begin{aligned}
\aaE(f)=\int_{\R^{2d}} \frac{|v|^2}{2} f\dd v\dd x\quad&\text{and}\quad \aS(f)=\cH(f),\\
\aM(f)g=-\pi\overline{\nabla}\cdot\Big(|V|^2ff'f_* f'_*\overline{\nabla} g\Big)\quad&\text{and}\quad\aL(f)g=-\nabla\cdot(f \aJ\nabla g).
\end{aligned}
\end{equation}
Similar to Section \ref{sub-sec:3}, one can check $(\aaE,\aS,\aM,\aL)$ is indeed a building block for the four-wave kinetic equation \eqref{4WKE}.
%\begin{remark}
%As a consequence of the GENERIC formulations for the 3-wave and 4-wave kinetic equations,  by simply letting $\aL=0$ and $\aaE=0$, it follows that the corresponding spatially homogeneous models are gradient flows.
%\end{remark}
As has been mentioned in the introduction, the WKE and Boltzmann equations have many similar properties. In the following remark, we compare them in view of their GENERIC structures.
\begin{remark}[Comparison with Boltzmann equation]\label{rmk:boltzmann}
The kinetic Boltzmann equation can be written as 
    \begin{equation}
\label{Boltzmann}
    \left\{
    \begin{aligned}
    &(\d_t+v\cdot\nabla_x)f=Q(f)\\
    &Q(f)=\int_{\R^d\times  S^{d-1}}|V|^2(f'f'_*-ff_* )\dd\sigma \dd v_*. 
    \end{aligned}\right.
\end{equation}

{ In this remark, we compare the four-wave kinetic equation \eqref{WKE} with $\omega(v)=|v|^2$ and the Boltzmann equation at a formal level. 
In the kinetic theory, the Boltzmann collision kernel in \eqref{Boltzmann} normally takes the form
\begin{equation}
    \label{kernerl:boltzmann}
    |V|^2=|v-v_*|^\gamma b(\theta),\quad \gamma\in[-2,1],
\end{equation}
where $b\ge 0$ and $\theta=\arccos\frac{\langle v-v_*,\sigma\rangle}{|v-v_*|}$ denotes the deviation angle. Here, we consider a general kernel $V$ satisfying the symmetry condition
\begin{align}
\label{sys:V-*}
V(v,v',v_*,v_*')=V(v',v,v'_*,v_*)=V(v_*,v_*',v,v').    
\end{align} Notice that, as a consequence of the momentum and energy conservation laws \eqref{cl}, the Boltzmann kernel satisfies the symmetry condition \eqref{sys:V-*}.  For such Boltzmann collision kernels, the four-wave kinetic equation is studied in \cite{QSW25}.}

The Boltzmann equation \eqref{Boltzmann} can be seen as a GENERIC system of the building block  \eqref{block-4} with $\aS$ and $\aM$ replaced by
    \begin{equation*}
\label{block-boltzmann}
\aS(f)=-\int_{\R^{2d}}f\log f\dd v\dd x,\quad\aM(f)g=-\overline{\nabla}\cdot\Big(|V|^2\Lambda(f)\overline{\nabla} g\Big),
\end{equation*}
where the divergence operator is defined in \ref{nabla-4}, and $\Lambda(f)=\Lambda(ff_*,f' f'_*)$, and $\Lambda(a,b)=\frac{a-b}{\log a-\log b}$ denotes the { logarithmic} mean of $a,\,b>0$. The entropy functional
\begin{align*}
    \cH_B(f)=\int f\log f
\end{align*}
is so-called the Boltzmann entropy.

Notice that $\dd\aS$ is given by $\dd\aS(f)=-(\log f+1)$. 
The Boltzmann equation \eqref{Boltzmann} can be written as 
\begin{equation*}
(\d_t+v\cdot\nabla_x)f=\frac{1}{4}\overline{\nabla}\cdot\big({ |V|^2}\Lambda(f)\overline{\nabla}\log f\big).
\end{equation*}
The building block was shown in \cite{Ott97} (see also \cite{grmela2018generic} for an alternative construction based on non-quadratic dissipation potentials), and has been further studied in \cite{Erb23,EH25}.

We now discuss the similarities and differences between the GENERIC structures of the Boltzmann equation and the wave kinetic equation \eqref{WKE}. 
\begin{itemize}
    \item Similarities: the first similarity is that they share the same reversible dynamics (including the energy functional and the Poison operator) are the same. This is because we deal with the dispersion relation $\omega(v)=|v|^2$. The second similarity, which is intriguing, is that the discrete gradient and divergence operators appearing in the definition of the dissipative parts are also the same.
    \item Differences: a noticeable difference between the dissipative parts of the WKE and the Boltzmann equation is the entropy functional, which is 
    \begin{equation}
    \label{entropy functionals}    
    \begin{aligned}
    & \aS_1(f)=\int \log f\text{ for the WKE equation and}\\
    & \aS_2(f)=-\int f\log f\quad\text{for the Boltzmann equation}.
    \end{aligned}
    \end{equation}
    In addition the divergence operators are also weighted differently, which reflects the difference in the collision operators, namely wave collisions in the WKE and particle collisions in the Boltzmann equation. We provide further discussion on the entropy functionals in the next remark.
\end{itemize}
\end{remark}
In the following remark, we compare the dissipative operator for the WKE and Boltzmann equation with that of the heat/diffusion equation.
\begin{remark}[Comparison with the heat equation]\label{rmk:heat}
It is well-known that the PDE 
\[
\partial_t f=\Delta f
\]
can be viewed as either the heat equation describing the heat conduction phenomena or the diffusion equation modelling the diffusion phenomena. The difference between the twos can be recognised via the corresponding gradient flow structure \cite{JKO98,peletier2014large}
\begin{align*}
    \d_t f=\aM_i(f) \dd\aS_i,\quad i=1,2 
\end{align*}
where $\aS_1$ and $\aS_2$ are respectively the entropy functional for the WKE and Boltzmann equation defined in \eqref{entropy functionals} and the dissipative operator are given by
\begin{align*}
\aM_1(f)g=-\nabla\cdot(f^2\nabla g)\quad\text{and}\quad  \aM_2(f)g=-\nabla\cdot(f \nabla g).
\end{align*}
We recall that the operator $\aM_2$ induces to the Wasserstein metric \cite{JKO98}. {To see how $(\aM_i, \aS_i)$ gives the interpretation of heat conduction for $i = 1$ and the interpretation of diffusion for $i = 2$ one needs to look at the origin of these structures at the microscopic levels. Accordingly, the gradient flow structure $(\aM_1,\aS_1)$ can be derived from a large deviation principle for the empirical measure of the so-called Brownian Energy Processes that models the conduction of heat \cite{peletier2014large}. On the other hand, the Wasserstein gradient flow structure $(\aM_2,\aS_2)$ can be derived from the large deviation principle of the empirical measure of a collection of Brownian motions, which models the diffusion process \cite{ADPZ11,duong2013wasserstein}.} 

It is interesting that the WKE share similar dissipative structure, in particular the same entropy functional, as that of the heat equation \cite{peletier2014large}. This intriguing similarity may contain deeper connections between the two equations and demands further investigations {{by connecting these GENERIC structures to many-particle systems at the microscopic levels. We refer to \cite{guioth2022path,onuki2024dynamical} for recent works on this direction.}}

\end{remark}
The next two remarks discuss the GENERIC structure for the four-wave equation with general dispersion relation and the parametrisation of the resonant manifold in the three-wave equation.
\begin{remark}[General dispersion relation]
Similar to the three-wave kinetic equation \eqref{3WKE}, one can also define $\overline\nabla$ in the non-parametrised setting, where we define
\begin{align*}
\overline\nabla\phi=\phi_1+\phi_3-\phi_2-\phi\in\R^{4d},  \end{align*}
and the wave kinetic equation can be written as
\begin{align*}
    (\d_t+v\cdot\nabla_x)f=-\pi\overline{\nabla}\cdot \big(\delta_0\hat\delta_0 |V|^2ff_1f_2f_3\overline\nabla(f^{-1})\big),
\end{align*}
where we use the notations $\delta_0=\delta_d(v+v_2-v_1-v_3)$ and $\hat\delta_0=\delta_1(\omega+\omega_2-\omega_1-\omega_3)$.  
The same GENERIC formulation holds with the operator $\aM$ replaced by $\aM(f)g=-\pi\overline{\nabla}\cdot\Big(\delta_0\hat\delta_0|V|^2ff_1f_2 f_3\overline{\nabla} g\Big)$.
\end{remark}

\begin{remark}[Parametrisation of three-wave kinetic equations]\label{rmk:3}
For the three-wave kinetic equation \eqref{3WKE}, we take  $\omega(v)=|v|^2$. { For a pair $(v,v_1,v_2)$ satisfying $v=v_1+v_2$ and $|v|^2=|v_1|^2+|v_2|^2$, we write $v_+=v_1$ and $v_-=v_2$.} Notice that $v^+\cdot v^-=0$, and $v^+$ and $v^-$ can be parametrised as follows
\begin{align*}
v^+=\frac{v}{2}+\frac{|v|}{2}\sigma,\quad v^-=\frac{v}{2}-\frac{|v|}{2}\sigma,\quad \sigma\in S^{d-1}.   
\end{align*}

Following \eqref{**} and \eqref{*} analogously, we define $\overline\nabla_3$ and $\overline\nabla_3\cdot$ as follows
\begin{align*}
&\overline\nabla_3\phi=\phi_++\phi_--\phi\quad\text{and}\\
&\overline\nabla_3\cdot G(x,v)=\int_{S^{d-1}}G(x,v,v_+,v_-)-G(x,v_+,v,v_-)-G(x,v_-,v_+,v)\dd\sigma
\end{align*}
for $\phi=\phi(x,v)$ and $G=G(x,v,v_+,v_-)$.
However, we lack the following identity
\begin{align*}
(v_+)_+= v\text{ or }v_-\quad\text{and}\quad(v_-)_-= v\text{ or }v_+, 
\end{align*}
and hence, there's no integration by parts identity 
\begin{align*}
\int_{\R^{2d}}(\overline\nabla_3\cdot G)\phi\dd v\dd x\neq\int_{\R^{2d}\times S^{d-1}}G(\overline\nabla_3\phi)\dd\sigma\dd v\dd x.    
\end{align*}
In other word, different from four-wave case, the weak formulation of $Q_3(f)$ is not invariant under the transformation $(v,v_\pm,v_\mp)\not\mapsto (v_\mp,v,v_\pm)$, and we can not bring $\delta_d(v-v_1-v_2)\delta_1(\omega-\omega_1-\omega_2)$ inside of the definition of $\overline\nabla_3$ to write $Q_3(f)$ into a divergence-form.
\end{remark}

{ \begin{remark}
   
In this remark, we review relevant works in the literature on the existence theory for the kinetic wave equations and discuss some of key challenges in making the GENERIC (gradient flow) structure rigorous. In  \cite{escobedo2015theory}, the authors prove local and global well posedness results for suitable concepts of weak and mild solutions of the isotropic four-wave WKE that arises in the weak turbulence theory for the cubic nonlinear Schrödinger equation in which the dispersion relation is quadratic $\omega(v)=|v|^2$ and the collision kernel is the identity $(V=1)$. Similar results for more general radial dispersion relations and in nearly critical weighted spaces $L^2$ and $L^\infty$ are obtained in \cite{germain2020optimal}. It is then extended to all almost critical weighted $L^r$ spaces for $2\le r\le\infty$ in the recent work \cite{ampatzoglou2025optimal}.

The formal calculations in this paper strongly suggests a variational GENERIC structure (gradient flows for the homogeneous model) for the WKEs. It would be desirable to establish this rigorously.  The GENERIC structures for the Boltzmann and Landau equations have been proved recently in \cite{Erb23,EH25,carrillo2024landau,DuongHe2025} via a variational characterisation in the $L^1$-framework. However, this framework does not directly apply to the WKE. The first reason is that, as shown in \cite{escobedo2015theory,germain2020optimal,ampatzoglou2025optimal}, $L^1$ is a supercritical space for the WKEs. The second one is to establish the integrability of the entropy dissipation functional as in the framework of $\mathcal{H}$-solutions developed by \textcite{villani1998new}. The strategy used in \cite{Erb23,EH25,carrillo2024landau,DuongHe2025} for the Boltzmann equation relies on  the associated (transport) collision rate equation
\begin{align*}
    \d_t f+v\cdot \nabla_x f=\overline\nabla \cdot U,
\end{align*}
where $U$ denotes a collision rate, to rigorously show the chain rule
\begin{align*}
    \frac{\dd}{\dd t} \cH_B(f)=\int \overline\nabla \log f U.
\end{align*}
The key idea in these papers is to use the convexity of the entropy dissipation and the curve action 
\begin{align*}
\cD_{B}=\int |V|^2\Lambda(ff_*,f'f_*')|\overline\nabla \log f|^2,\quad \mathcal{A}_{B}=\int \frac{|U|^2}{|V|^2\Lambda(ff_*,f'f_*')}
\end{align*}
to employ regularisation techniques and pass to the limits. However, in the WKE case, due to the multiplicative structure of the collision kernel, we lack of the convexity of the corresponding entropy dissipation and curve action
\begin{equation*}
\cD_{WKE}=\int |V|^2  ff_*f'f_*' \overline\nabla f^{-1},\quad   \cA_{WKE}=\int\frac{|U|^2}{|V|^2ff_*f'f_*'}
\end{equation*}
Among other things, these two reasons pose significant challenges for a rigorous proof of the GENERIC structure for the WKEs, which demands further works in the future. We also refer to  a recent work \cite{EGLM25} on the static problems where the authors study the maximisers for the entropy functional on the torus.

\end{remark}}

%{\color{green}We comment on showing GENERIC structure rigorously. The GENERIC (gradient flow) structures Boltzmann and Landau equations  have been studied in XXXX via a variational characterisation with probability density curves.  On the one hand, the WKE has also conerved mass, however, the collision terms are much more sigular than Boltzmann equations. This leads to difficultyies to study well-posedness results. TO our best knowledge, there's local-in-time $L^1$ solutions in XXXX. The GENERIC framework based on a $\cH$-Theorem for WKE. WKE has at least formally equilirium $M(v)=\frac{1}{a+b\cdot v+c\omega(v)}$. We want the dissipation of  
%\begin{align*}
%    &\cH(f|M):=\int_{\R^d}\log\frac{f}{M}\dd v,\\
%    &\frac{\dd }{\dd t}\cH(f|M)=-\cD(f)\le0 .
%\end{align*}
%To study the class of solutions $|\cH(f|M)|<+\infty$, a natural class of solutions are $0<c\le \frac{f}{M}\le C<+\infty$ and $\frac{f}{M}-1\in L^1(\R^d)$. Such class of solutions are still lack of study. 
%BOltzmann and Landau is studied in $L^1$-framework. However, for 4 wke, $L^1$ is super critical spce, there are only XXXX, we need different techniques and framework. 

%So we leave the rigourous proof for the furture work.}

\section{Small-angle limit}
\label{sec:limit}
In this section, we derive a small-angle limit of the four-wave equation with dispersion relation $\omega(v)=|v|^2$
\begin{equation}
\label{eq-4}
(\d_t+v\cdot\nabla_x)f=-4\pi\int_{\R^d\times S^{d-1}}|V|^2ff_* f'f'_*\overline\nabla f^{-1}\dd\sigma \dd v_*.
\end{equation}
As mentioned earlier, this is inspired by the grazing limit from the Boltzmann equation to the Landau equation.

We recall that  $\overline{\nabla}$ is given by $\overline\nabla\phi=\phi'+\phi'_*-\phi_*-\phi$, and
\begin{align*}
v'=\frac{v+v_*}{2}+\frac{|v-v_*|}{2}\sigma,\quad v_*'= \frac{v+v_*}{2}-\frac{|v-v_*|}{2}\sigma.
\end{align*}
We define the deviation angle 
\begin{align*}
\theta=\arccos \frac{v-v_*}{|v-v_*|}\cdot \sigma.
\end{align*}
We assume that the kernel $|V|^2$  has the form
\begin{equation*}
\label{kernel-B}
|V|^2=B(|v-v_*|)b(\theta)\ge 0, \quad \theta\in[0,\pi/2].
\end{equation*}
{ Analogously to the Boltzmann equation, here we take 
$B(|v-v_*|)=|v-v_*|^\gamma$ with $\gamma\in[-2,1]$.}
One can restrict $\theta$ on $[0,\pi/2]$ by symmetrising $V=V(|v-v_*|,\sigma)$ as $\big(V(|v-v_*|,\sigma)+V(|v-v_*|,-\sigma)\big)\mathbb{1}_{\theta\in [0,\pi/2]}$.
 We will show that if the wave interaction happens when $\theta\sim0$, the wave kinetic equation \eqref{eq-4} formally converges to 
\begin{equation}
\label{Landau-2} 
\left\{
\begin{aligned}
&(\d_t +v\cdot\nabla_x)f=Q_L(f)\\
&Q_L(f)=-4\pi\nabla_v\cdot\int_{\R^{d}}B_0(ff_*)^2\Pi_{(v-v_*)^\perp}\big(\nabla_v f^{-1}-(\nabla_{v}f^{-1})_*\big) \dd v_*,      
\end{aligned}
\right.
\end{equation}
{ where the kernel
\begin{align*}
B_0(|v-v_*|):=B(|v-v_*|)|v-v_*|^2,    
\end{align*}}
and $\Pi_{v^\perp}$ denotes the projection onto $v^\perp$ and is given by
\begin{align*}
\Pi_{v^\perp}=\operatorname{Id}-\frac{v\otimes v}{|v|^2}.
\end{align*}
In Section \ref{sec-3:scaling}, we present the small-angle scaling and compare it with the Boltzmann grazing limit. 
In Section \ref{sec-3:limit}, we rigorously derive the small-angle limit of the rescaled collision operator obtaining the limiting operator in the equation \eqref{eq-4}. In Section \ref{sec-3:generic}, we show that this limiting equation \eqref{Landau-2} also has a GENERIC formulation similar to that of the Landau equation. Thus, we  establish a new analogy between the wave turbulence theory and classical kinetic theory.
\subsection{Small-angle scaling and Boltzmann grazing limit}
\label{sec-3:scaling}

We first present the small-angle scaling.
We define the vector
\begin{align*}
k=\frac{v-v_*}{|v-v_*|}\in S^{d-1}.
\end{align*}
We recall that the deviation angle $\theta=\arccos k\cdot\sigma$.
We define the angle function
\begin{align*}
\beta(\theta)=\sin\theta^{d-2}b(\theta),
\end{align*}
and it satisfies $\supp(\beta)\subset [0,\pi/2]$, $\beta(\theta)\gtrsim \theta^{-2} $, and the following angular moment assumption
\begin{align}
\label{int:beta}
 \int_{0}^{\frac{\pi}{2}}\theta^2\beta(\theta)\dd\theta=8(d-1)/|S^{d-2}|,
\end{align}
where the constant on the right-hand side is chosen for renormalisation.
For $d=2$, we take $|S^0|=2$.
For any $\varepsilon\in(0,1)$, we take a sequence of scaling
\begin{equation*}
\beta^\varepsilon(\theta)=\frac{\pi^3}{\varepsilon^3} \beta (\pi\theta/\varepsilon).
\end{equation*}
Then, for any $\varepsilon\in(0,1)$, we have
\begin{align*}
  \supp (\beta^\varepsilon)\subset[0,\varepsilon/2]\quad\text{and}\quad  \int_{0}^{\frac{\varepsilon}{2}}\theta^2\beta^\varepsilon(\theta)\dd\theta=8(d-1)/|S^{d-2}|.
\end{align*}
Correspondingly, we define the scaling kernels
\begin{equation*}
\begin{aligned}
&b^\varepsilon(\theta)=(\sin\theta)^{-(d-2)}\beta^\varepsilon(\theta)\quad\text{and}\quad |V^\varepsilon|^2= B (|v-v_*|)b^\varepsilon(\theta)\label{kernel:epsilon}.
\end{aligned}
\end{equation*}

For $d\ge3$, for a given $k\in S^{d-1}$, we define the sphere $S^{d-2}_{k^\perp}=\{p\in S^{d-1}\mid k\cdot p=0\}$. For $d=2$ and $k=(k_1,k_2)\in S^1$, we use the notation $S^{0}_{k^\perp}=\{(k_2,-k_1),\,(-k_2,k_1)\}$, and  $ \int_{S^{0}_{k^\perp}}f=f(k_2,-k_1)+f(-k_2,k_1)$. The scaling collision term $Q^\varepsilon(f)$ can be written as
\begin{equation}
    \label{qs-4}
\begin{aligned}
    Q^\varepsilon(f)&=-4\pi\int_{\R^{d}\times S^{d-1}}|V^\varepsilon|^2ff_*f'f_*'\overline\nabla f^{-1}\dd\sigma\dd v_*\\
    &=-4\pi\int_{\R^{d}}\int_0^{\varepsilon/2}\int_{S^{d-2}_{k^\perp}}B\beta^\varepsilon(\theta)ff_*f'f_*'\overline\nabla f^{-1}\dd p\dd \theta\dd v_*.
\end{aligned}
\end{equation}

In Section \ref{sec-3:limit}, we rigorously show that 
$$Q^\varepsilon(f)\to Q_{L}(f) \quad\text{as}\quad \varepsilon\to 0.$$ We emphasise that the convergence of the collision operators does not by itself imply convergence of the corresponding dynamics; additional uniform estimates and compactness arguments are required. The same caveat applies to the Boltzmann-to-Landau limit.
%This scaling four-wave kinetic equation \eqref{qs-4} converges to \eqref{Landau-2} at least on the formal by showing 
%Notice that the convergence of the collisional operators does not straightforwardly imply the convergence of the dynamics, which requires precise compactness of the sequence of solutions $f^\varepsilon$ corresponding to $\d_t f^\varepsilon=Q^\varepsilon(f^\varepsilon)$.}
A similar small-angle (grazing) limit has been studied for (kinetic) Boltzmann equations in literature, see  \cite{Des92,Vil96,AV04,carrillo2022boltzmann}. More precisely, as 
 $\varepsilon\to0$ 
 the Boltzmann equation
\begin{align*}
    (\d_t+v\cdot \nabla_x)f=\int_{\R^d\times S^{d-1}}|V^\varepsilon|^2(f'f_*'-ff_*)\dd\sigma\dd v_*
\end{align*}
converges to the Landau equation
\begin{equation*}
 \label{Landau}  
 \begin{aligned}
    (\d_t +v\cdot\nabla_x)f=\nabla_v\cdot \int_{\R^d}B_0 ff_*\Pi_{(v-v_*)^\perp}\big(\nabla_v \log f-\nabla_{v_*}\log f_*\big)\dd v_*.
\end{aligned}
\end{equation*}
To the best of our knowledge, the limiting equation \eqref{Landau-2} has not been derived in the literature. However, it is worth mentioning that the small-angle limit, combined with the limit $v\sim v_*$ for the four-wave kinetic equation has been studied in \cite{HH81}.
{ \subsection{Convergence of collision operators}
\label{sec-3:limit}
We recall the collision operators
\begin{equation*}
    \label{op:e-L}
\begin{aligned}
   Q^\varepsilon(f)&=- 4\pi\int_{\R^{d}}\int_0^{\varepsilon/2}|v-v_*|^\gamma\beta^\varepsilon(\theta)ff_*\Big(\int_{S^{d-2}_{k^\perp}}f'f_*'\overline\nabla f^{-1}\dd p\Big)\dd \theta\dd v_*,\\
   Q_L(f)&= -4\pi\nabla_v\cdot\int_{\R^{d}}|v-v_*|^{\gamma+2}(ff_*)^2\Pi_{(v-v_*)^\perp}\big(\nabla_v f^{-1}-\nabla_{v_*}f^{-1}_*\big) \dd v_*.
\end{aligned}
\end{equation*}
Let $W^{1,1}_{s}(\R^d)$ denote the functional space that contains all functions $\xi=\xi(v)$ such that $\langle v\rangle^s\xi\in W^{1,1}(\R^d)$. Let $\langle\cdot,\cdot \rangle$ denote the $L^2$-inner product on $\Do$.

In this Section, we show the following convergence results.
\begin{theorem}
\label{thm:limit}
Let $\gamma\in[-2,1]$. Suppose that $f\in L^3\big(\R^d_x;W^{1,1}_{\gamma+2}\cap W^{1,\infty}(\R^d_v)\big)$. 
As $\varepsilon\to0$, we have 
\begin{equation}
\label{limit:weak}
\langle Q^\varepsilon(f),\varphi \rangle\to \langle Q_L(f),\varphi \rangle
\end{equation}
for all $\varphi\in C^\infty_c(\Do)$.
\end{theorem}

We use the following lemma to show  Theorem \ref{thm:limit}.
\begin{lemma}
\label{lem:limit}
Let $\varepsilon\in(0,1)$. Let $\xi$ and $\zeta\in W^{1,\infty}(\R^d;\R_+)$ and  $\phi\in W^{3,\infty}(\R^d)$. 
\begin{itemize}
    \item For all $(v,v_*)\in\Do$ and $\theta\in [0,\varepsilon/2]$, the following pointwise bounds hold
    \begin{equation}
        \label{bound-upper}
\begin{aligned}
&\Big|\int_{S^{d-2}_{k^\perp}}\xi'\zeta_*'\overline\nabla\phi\dd p\Big|\\
\lesssim&{}\theta^2|v-v_*|^2\|\phi\|_{W^{2,\infty}}\Big(\xi\zeta_*+\int_{S^{d-2}_{k^\perp}}\int_0^1\zeta_*'|\nabla_v\xi(\bar v)|+\xi'|\nabla_{v_*}\zeta(\bar v_*)|\dd t\dd p\Big),
\end{aligned}
\end{equation}
where $k:=\frac{v-v_*}{|v-v_*|}\in S^{d-1}$ and 
\begin{equation}
\label{def:v-bar}
\bar v:=\frac{v+v_*}{2}+\frac{|v-v_*|}{2}\big(t\sigma+(1-t)k\big),\quad \bar v_*:=\frac{v+v_*}{2}-\frac{|v-v_*|}{2}\big(t\sigma+(1-t)k\big).
\end{equation}

\item Let $\chi=\pi\theta/\varepsilon\in[0,\pi/2]$. We have the  following pointwise convergence
\begin{equation}
\label{limit}
\lim_{\varepsilon\to0}\frac{\pi^2}{\varepsilon^2}\int_{S^{d-2}_{k^\perp}}\xi'\zeta_*'\overline\nabla\phi\dd p=\chi^2q_L(\xi,\zeta,\phi),
\end{equation}
where $q_L$ is defined as 
\begin{align*}
    q_L(\xi,\zeta,\phi):=&\frac{|S^{d-2}|}{8(d-1)}|v-v_*|^2\Big(2(\nabla_v -\nabla_{v_*})\xi\zeta_*\cdot \Pi_{(v-v_*)^\perp}\big(\nabla_v\phi-\nabla_{v_*}\phi_*\big)\\
   &+\xi\zeta_*(\nabla_v-\nabla_{v_*})\cdot\Pi_{(v-v_*)^\perp}\big(\nabla_v \phi-\nabla_{v_*}\phi_*\big)\Big).
\end{align*}
\end{itemize}
\end{lemma}
We note that the factor $\chi^2$ on the right-hand side of \eqref{limit} is of order $1$, which will be used, together with the angular momentum assumption \eqref{int:beta}, to handle the angular singularity.
\begin{proof}[Proof of Lemma \ref{lem:limit}]
 The proof follows \cite{Vil96,carrillo2022boltzmann} with appropriate modification for the wave collision operators.

We define
\begin{align*}
    v_1:=\frac{v+v_*}{2}\quad\text{and}\quad v_2:=\frac{v-v_*}{2}.
\end{align*}
Notice that
\begin{align*}
    v=v_1+|v_2|k,\quad  v=v_1-|v_2|k,\quad  v'=v_1+|v_2|\sigma,\quad  v_*=v_1-|v_2|\sigma.
\end{align*}

By Taylor expansion, we have
\begin{equation}
    \label{taylor-1}
\begin{aligned}
    \phi'&=\phi+|v_2|(\sigma-k)\cdot \int_0^1\nabla_v\phi(\bar v)\dd t,\\
    \phi_*'&=\phi_*-|v_2|(\sigma-k)\cdot \int_0^1\nabla_{v_*}\phi(\bar v_*)\dd t,
\end{aligned}
\end{equation}
where $\bar v$ and $\bar v_*$ are given in \eqref{def:v-bar}. 

Applying \eqref{taylor-1} to functions $\xi$ and $\zeta$, we have 
\begin{align*}
    \xi'\zeta_*'&=\xi\zeta_*+\underbrace{|v_2|(\sigma-k)\cdot \int_0^1\zeta_*'\nabla_v\xi(\bar v)-\xi'\nabla_{v_*}\zeta(\bar v_*)}_{=:I_1}.
\end{align*}

Similarly, by Taylor expansion, we have 
\begin{equation*}
    \label{taylor-2}
\begin{aligned}
\overline\nabla\phi&= \phi'-\phi+ \phi'_*-\phi_*\\
 &=\underbrace{|v_2|(\sigma-k)\cdot\big(\nabla_v\phi-\nabla_{v_*}\phi_*\big)}_{=:I_2}+\underbrace{|v_2|^2(\sigma-k)\otimes (\sigma-k):T}_{=:I_3},
\end{aligned}
\end{equation*}
where $v_{\pm}:=v_1\pm |v_2|\big(s(t\sigma+(1-t)k)+(1-s)k\big)$ and 
\begin{align*}
T&:=\int_0^1t\int_0^1\frac{D^2\phi(v_+)+D^2\phi(v_-)}{2}\dd s\dd t.
\end{align*}

\medskip

By definition of $\theta$, we have 
\begin{align}
\label{sigma-k}
\sigma-k=k(\cos\theta-1)+p\sin\theta,\quad p=\Pi_{k^\perp}\sigma\in S^{d-2}_{k\perp}.   \end{align}
For $\theta\sim\varepsilon$ small, we have 
\begin{align*}
\cos\theta-1=-\frac{\theta^2}{2}+o(\theta^2)\quad\text{and}\quad   \sin\theta=\theta+o(\theta).  
\end{align*}
Then the following pointwise bounds hold
\begin{gather}
    |I_1|\lesssim \theta|v_2|\|\xi\|_{W^{1,\infty}}\|\zeta\|_{W^{1,\infty}}, \quad 
  |I_2|\lesssim \theta|v_2|^2 \|D^2_v\phi\|_{L^\infty}, \quad |I_3|\lesssim \theta^2|v_2|^2 \|D^2_v\phi\|_{L^\infty},\label{bound:1-2-3}
\end{gather}
where we use $|\nabla_v\phi-\nabla_{v_*}\phi_*|\lesssim |v_2|\|D^2_v\phi\|_{L^\infty}$ to obtain the quadratic bound $|v_2|^2$ for the upper bound of $I_2$.

Notice that, for $\theta\sim\varepsilon$ small,
\begin{align}
\label{I-0}
\lim_{\varepsilon\to0}\frac{\pi^2}{\varepsilon^2}\int_{S^{d-2}_{k^\perp}}\xi'\zeta_*'\overline\nabla\phi\dd p
=\chi^2\int_{S^{d-2}_{k^\perp}}I_1I_2+\xi\zeta_*(I_2+I_3)\dd p.
\end{align}
We show that $\int_{S^{d-2}_{k^\perp}}I_2\sim\theta^2$. By \eqref{sigma-k}, we have
\begin{equation*}
\begin{aligned}
    &\int_{S^{d-2}_{k^\perp}}(\sigma-k)\cdot\big(\nabla_v\phi-\nabla_{v_*}\phi_*\big)\dd p=\int_{S^{d-2}_{k^\perp}}\big((\cos\theta-1)N_k+\sin\theta N_p\big)\dd p,
\end{aligned}
\end{equation*}
where $N_k$ and $N_p$ denote the projections 
\begin{align*}
N_k=(\nabla_v\phi-\nabla_{v_*}\phi_*)\cdot k\quad\text{and}\quad  N_p=(\nabla_v\phi-\nabla_{v_*}\phi_*)\cdot p. 
\end{align*}
Notice that $N_k$ is independent of $p$, and $\int_{S^{d-2}_{k^\perp}} N_p\dd p=0$. Then we have 
\begin{gather}
    \int_{S^{d-2}_{k^\perp}}I_2\dd p=(\cos\theta-1)|v_2||S^{d-2}|(\nabla_v\phi-\nabla_{v_*}\phi_*)\cdot k,\label{int-1-1}\\
    \text{and}\quad \Big|\int_{S^{d-2}_{k^\perp}}I_2\dd p\Big|\lesssim \theta^2|v_2|\|\phi\|_{W^{1,\infty}}.\label{bound-2-int}
\end{gather}
Combining \eqref{bound:1-2-3} and \eqref{bound-2-int}, we show  \eqref{bound-upper}
\begin{align*}
&\Big|\int_{S^{d-2}_{k^\perp}}\xi'\zeta_*'\overline\nabla\phi\dd p\Big|\\
&\lesssim\theta^2|v-v_*|^2\|\phi\|_{W^{2,\infty}}\Big(\xi\zeta_*+\int_{S^{d-2}_{k^\perp}}\int_0^1\zeta_*'|\nabla_v\xi(\bar v)|+\xi'|\nabla_{v_*}\zeta(\bar v_*)|\dd t\dd p\Big).
\end{align*}

We are left to show the pointwise limit \eqref{I-0}.
\begin{itemize}
 \item {\bf \big($\int_{S^{d-2}_{k^\perp}}I_1I_2$ term\big).} Notice that 
\begin{align*}
\int_{S^{d-2}_{k^\perp}} (\sigma-k)\otimes(\sigma-k)\dd p=\theta^2\int_{S^{d-2}_{k^\perp}} p\otimes p\dd p+e(\theta),
\end{align*}
for some $e(\theta)\in \R^{d\times d}$ and $|e(\theta)|\lesssim \theta^3$.  Moreover, the following identity holds, see, for example \cite[Proposition 3.2]{DGH25}
\begin{equation*}
\label{perp}
\begin{aligned}
\int_{S^{d-2}_{k^\perp}} p\otimes p =\frac{|S^{d-2}|}{d-1}\Pi_{k^\perp}.
\end{aligned}
\end{equation*} 
Then we have 
    \begin{equation}
    \label{int-3}
    \begin{aligned}
        &\lim_{\varepsilon\to0}\frac{\pi^2}{\varepsilon^2} \int_{S^{d-2}_{k^\perp}}I_1 I_2\dd p\\
        =&{}\chi^2|v_2|^2\int_{S^{d-2}_{k^\perp}}p\otimes p\dd p:\big(\nabla_v\phi-\nabla_{v_*}\phi_*\big)\otimes (\nabla_v -\nabla_{v_*})\xi\zeta_*\\
        =&{}\chi^2|v_2|^2\frac{|S^{d-2}|}{d-1}\Pi_{k^\perp}:\big(\nabla_v\phi-\nabla_{v_*}\phi_*\big)\otimes (\nabla_v -\nabla_{v_*})\xi\zeta_*\\
        =&{}\chi^2|v_2|^2\frac{|S^{d-2}|}{d-1}(\nabla_v -\nabla_{v_*})\xi\zeta_*\cdot \Pi_{k^\perp}\big(\nabla_v\phi-\nabla_{v_*}\phi_*\big).
    \end{aligned}
    \end{equation}
    \item {\bf \big($\int_{S^{d-2}_{k^\perp}}I_2+I_3$ term\big).} Concerning $I_2$ term, the equality  \eqref{int-1-1} and $\cos\theta-1\sim-\theta^2/2$ imply that 
\begin{align}
   \lim_{\varepsilon\to0}\frac{\pi^2}{\varepsilon^2} \int_{S^{d-2}_{k^\perp}}I_2\dd p&=-\chi^2\frac{|S^{d-2}|}{2}(\nabla_v\phi-\nabla_{v_*}\phi_*)\cdot v_2.\label{int-1}
    \end{align}
Next we deal with $I_3$ term. 
Since $|T|\lesssim \|D^2_v\phi\|_{L^\infty}$ and $T\to \frac{D^2_v\phi+D^2_{v_*}\phi_*}{2}$, we repeat the argument as in \eqref{int-3} and use the dominated convergence theorem to obtain
\begin{equation}
\label{int-2}
\begin{aligned}
 \lim_{\varepsilon\to0}\frac{\pi^2}{\varepsilon^2} \int_{S^{d-2}_{k^\perp}}I_3\dd p
    &=\chi^2\frac{|v_2|^2|S^{d-2}|}{d-1} \Pi_{k^\perp}:\frac{D^2_v\phi+D^2_{v_*}\phi_*}{2}.
\end{aligned}
\end{equation}
Combining \eqref{int-1} and \eqref{int-2}, we have 
\begin{align*}
&\lim_{\varepsilon\to0}\frac{\pi^2}{\varepsilon^2} \int_{S^{d-2}_{k^\perp}}(I_2+I_3)\dd p\\
&=\chi^2\frac{|S^{d-2}|}{2}\Big(-v_2\cdot \big(\nabla_v \phi-\nabla_{v_*}\phi_*\big)+\frac{ |v_2|^2}{d-1} \Pi_{k^\perp}:(D^2_v\phi+D^2_{v_*}\phi_*)\Big).
\end{align*}
By using $\nabla_{v_2}=\nabla_v-\nabla_{v_*}$ and  the following identity
\begin{align*}
\nabla_{v_2}\cdot(|v_2|^2\Pi_{k^\perp})=-(d-1)v_2,    
\end{align*}
we have 
\begin{equation*}
\begin{aligned}
&\big(\nabla_{v_2}\cdot(|v_2|^2\Pi_{v_2^\perp})\big)\cdot (\nabla_v \phi-\nabla_{v_*}\phi_*)+|v_2|^2 \Pi_{v_2^\perp}:(D^2_v\phi+D^2_{v_*}\phi_*)\\
=&{}\nabla_{v_2}\cdot\Big(|v_2|^2\Pi_{v_2^\perp}(\nabla_v \phi-\nabla_{v_*}\phi_*)\Big)\\
=&{}\frac{1}{4}(\nabla_v-\nabla_{v_*})\cdot\Big(|v-v_*|^2\Pi_{k^\perp})(\nabla_v \phi-\nabla_{v_*}\phi_*)\Big).
\end{aligned}
\end{equation*}
Hence, we get
\begin{equation}
    \label{int-4}
\begin{aligned}
&\lim_{\varepsilon\to0}\frac{\pi^2}{\varepsilon^2} \int_{S^{d-2}_{k^\perp}}(I_2+I_3)\dd p\\
=&{}\chi^2\frac{|S^{d-2}|}{8(d-1)}(\nabla_v-\nabla_{v_*})\cdot\Big(|v-v_*|^2\Pi_{k^\perp}(\nabla_v \phi-\nabla_{v_*}\phi_*)\Big).
\end{aligned}
\end{equation}
   
    \end{itemize}

\medskip 

It follows from \eqref{I-0},  \eqref{int-4} and  \eqref{int-3} that
\begin{align*}
   &\lim_{\varepsilon\to0}\frac{\pi^2}{\varepsilon^2}\int_{S^{d-2}_{k^\perp}}\xi'\zeta_*' \overline\nabla \phi\dd p \\
   =&\chi^2\frac{|S^{d-2}|}{8(d-1)}\Big(2|v-v_*|^2(\nabla_v -\nabla_{v_*})\xi\zeta\cdot \Pi_{(v-v_*)^\perp}\big(\nabla_v\phi-\nabla_{v_*}\phi_*\big)\\
   &+\xi\zeta(\nabla_v-\nabla_{v_*})\cdot\big(|v-v_*|^2\Pi_{(v-v_*)^\perp}(\nabla_v \phi-\nabla_{v_*}\phi_*)\big)\Big).
\end{align*}
This completes the proof of this Lemma.
\end{proof}
We are now in the position to prove Theorem \ref{thm:limit}.
\begin{proof}[Proof of Theorem \ref{thm:limit}]

We recall $\chi=\pi\theta/\varepsilon$. For any function $\varphi\in C^\infty_c(\Do)$, the following weak formulations hold
\begin{align*}
   \langle Q^\varepsilon(f),\varphi\rangle&=\pi\int_{\R^{3d}}\int_0^{\pi/2}|v-v_*|^\gamma\beta(\chi)\frac{\pi^2}{\varepsilon^2}\Big(ff_*\int_{S^{d-2}_{k^\perp}}(f'+f_*')\overline\nabla \varphi\dd p\notag\\
   &\quad -(f+f_*)\int_{S^{d-2}_{k^\perp}}f'f_*'\overline\nabla \varphi\dd p\Big)\dd \chi\dd v_*\dd v\dd x,\\
    \langle Q_L(f),\varphi\rangle&= 2\pi\int_{\R^{2d}}|v-v_*|^{\gamma+2}\big(-f_*^2\nabla_v f+f^2\nabla_{v_*}f_*\big)\\
    &\quad\cdot \Pi_{(v-v_*)^\perp}\big(\nabla_v \varphi-\nabla_{v_*}\varphi_*\big) \dd v_*\dd v\dd x. 
\end{align*}

By Lemma \ref{lem:limit}, we have the following pointwise convergence, 
as $\varepsilon\to0$,
\begin{align*}
     &\frac{\pi^2}{\varepsilon^2}\Big(ff_*\int_{S^{d-2}_{k^\perp}}(f'+f_*')\overline\nabla \varphi\dd p-(f+f_*)\int_{S^{d-2}_{k^\perp}}f'f_*'\overline\nabla \varphi\dd p\Big)\\
     \rightarrow &{}\chi^2\Big(ff_*\big(q_L(1,f,\varphi)+q_L(f,1,\varphi)\big)-(f+f_*)q_L(f,f,\varphi)\Big)\\
     =&{}\chi^2|v-v_*|^2\frac{|S^{d-2}|}{4(d-1)}\big(-f_*^2\nabla_v f+f^2\nabla_{v_*}f_*\big)\cdot \Pi_{(v-v_*)^\perp}\big(\nabla_v\phi-\nabla_{v_*}\phi_*\big)\\
   =:&{}2\chi^2|v-v_*|^2\overline q_L(f,\varphi).
\end{align*}
By using the angular momentum assumption 
\[
\int_0^{\pi/2}\chi^2\beta(\chi)\dd \chi={8(d-1)}/{|S^{d-2}|},
\]
we have
\begin{align*}
&2\pi\int_{\R^{3d}}\int_0^{\pi/2}\beta(\chi)\chi^2\dd \chi|v-v_*|^{\gamma+2}\overline q_L(f,\varphi)\dd v_*\dd v\dd x=\langle Q_L(f),\varphi\rangle.
\end{align*}

\medskip
We define
\begin{align*}
   A_1&:= \frac{\pi^2}{\varepsilon^2}ff_*\int_{S^{d-2}_{k^\perp}}(f'+f_*')\overline\nabla \varphi\dd p,\quad A_2:=\frac{\pi^2}{\varepsilon^2}(f+f_*)\int_{S^{d-2}_{k^\perp}}f'f_*'\overline\nabla \varphi\dd p.
\end{align*}
We will use the dominated convergence theorem to show \eqref{limit:weak}. To this end, we now show the upper bounds of $|A_1|$ and $|A_2|$.
\begin{itemize}
    \item  {\bf \big($A_1$ term\big).} By \eqref{bound-upper}, the following pointwise bound holds 
    \begin{align*}
&|A_1|\lesssim\chi^2|v-v_*|^2\|\phi\|_{W^{2,\infty}_v}\|f\|_{W^{1,\infty}_v}ff_*.
\end{align*}
Moreover, by assumption $2+\gamma\in[0,3]$, we have 
\begin{align*}
\int_{\R^{3d}}\int_0^{\pi/2}|v-v_*|^\gamma\beta(\chi)|A_1|&\lesssim\|\phi\|_{W^{2,\infty}}\int_{\R^{3d}}|v-v_*|^{\gamma+2}ff_*\|f\|_{W^{1,\infty}_v}\\
&\lesssim\|\phi\|_{W^{2,\infty}}\|f\|_{L^3_xL^1_{\gamma+2,v}}^2\|f\|_{L^3_xW^{1,\infty}_v}<+\infty.
\end{align*}

\item  {\bf \big($A_2$ term\big).} By \eqref{bound-upper}, the following pointwise bounds hold
\begin{align*}
&|A_2|\lesssim\chi^2|v-v_*|^2\|\phi\|_{W^{2,\infty}_v}(f+f_*)\Big(ff_*\\
&+\int_{S^{d-2}_{k^\perp}}\int_0^1f_*'|\nabla_vf(\bar v)|+f'|\nabla_{v_*}f(\bar v_*)|\dd t\dd p\Big)\\
&\lesssim\chi^2|v-v_*|^2\|\phi\|_{W^{2,\infty}_v}\big(ff_*+ff_*'+f'f_*\big)\|f\|_{W^{1,\infty}_v}\\
&+\chi^2|v-v_*|^2\|\phi\|_{W^{2,\infty}_v}\|f\|_{L^\infty_v}\int_{S^{d-2}_{k^\perp}}\int_0^1f_*|\nabla_vf(\bar v)|+f|\nabla_{v_*}f(\bar v_*)|\dd t\dd p,
\end{align*}
where $\bar v$ and $\bar v_*$ are defined as in \eqref{def:v-bar}.

{ The integrability of the first term on the right-hand side follows in the same way as for the $A_1$ term, where, for any fixed  $\sigma\in S^{d-1}$, we have 
\begin{align*}
&\int_{\R^{2d}}|v-v_*|^{\gamma+2}\big(ff_*+ff_*'+f'f_*\big)\dd v\dd v_*
\lesssim\|f\|_{L^1_{\gamma+2,v}}^2.
\end{align*}
One can deal with the integration of $|v-v_*|^{\gamma+2}f'f_*$ by changing of variables $v\mapsto v'$, since $|v-v_*|\lesssim |v'-v_*|$ and the Jacobian $\big|\frac{\d v}{\d v'}\big|\le 2^d$ (see below). The integration of $|v-v_*|^{\gamma+2}ff_*'$ follows analogously.
}

By symmetry, it therefore remains only to show the boundedness of 
\begin{equation}
\label{A-2}
\int_{\R^{3d}}|v-v_*|^{\gamma+2}\|f\|_{L^\infty_v}\int_0^{\pi/2}\beta(\chi)\chi^2\int_{S^{d-2}_{k^\perp}}\int_0^1f_*|\nabla_vf(\bar v)|.
\end{equation}
To this end, for any fixed pair $(t,\sigma,v_*)\in [0,1]\times S^{d-1}\times \R^d$,  we use a change of variables $v\mapsto \bar v$. We recall from \eqref{def:v-bar} that
\begin{align*}
&\bar v:=\frac{v+v_*}{2}+\frac{|v-v_*|}{2}\big(t\sigma+(1-t)k\big).
\end{align*}
Using these expressions, we can compute the Jacobian explicitly
\begin{align*}
 \Big|\frac{\d \bar v}{\d v}\Big|&=\det\Big(\big(1-\frac{t}{2}\big)
         \operatorname{I}+\frac{t}{2}\sigma \otimes k\Big)\\
%&=\det(\frac12 I+\frac12 t\sigma \otimes\frac{v-v_*}{|v-v_*|} +\frac12(1-t) I)\\
%&=\det((1-t/2) I+\frac12 t\sigma \otimes\frac{v-v_*}{|v-v_*|} )\\
&=\big(1-\frac{t}{2}\big)^{d-1}\Big(1+\frac{t}{2}(\sigma\cdot k-1)\Big)\\
&=\big(1-\frac{t}{2}\big)^{d-1}\Big(1-t\sin^2(\theta/2)\Big),\quad \theta\in [0,\pi/2].
\end{align*}
Therefore, we get
\begin{align*}
\Big|\frac{\d v}{\d \bar v}\Big|\le 2^{d}.
\end{align*}
Notice further that  $|v-v_*|\le 2|\bar v-v_*|$.
By changing of variables $v\mapsto \bar v$, we have 
\begin{align*}
&\int_{\R^{2d}}|v-v_*|^{\gamma+2}f_*|\nabla_vf(\bar v)|\dd v\dd v_*\\
\lesssim&{}\int_{\R^{2d}}|\bar v-v_*|^{\gamma+2}f(v_*)|\nabla_vf(\bar v)|\dd \bar v\dd v_*\lesssim \|f\|_{W^{1,1}_{\gamma+2,v}}^2.
\end{align*}
Hence, \eqref{A-2} is bounded by 
\begin{equation*}
\int_{\R^{d}}\|f\|_{L^\infty_v}\|f\|_{W^{1,1}_{\gamma+2,v}}^2\dd x\lesssim \|f\|_{L^3_xL^\infty_v}\|f\|_{L^3_xW^{1,1}_{\gamma+2,v}}^2.
\end{equation*}
This finishes the proof of this theorem.
\end{itemize}
   
\end{proof}}

In the following remark, we comment on the small-angle limit for the three-wave kinetic equation.
\begin{remark}[Small-angle limit for three-wave kinetic equation]
%The three-wave kinetic equation case is different from the four-wave kinetic equation one. 
We do the parametrisation as in Remark \ref{rmk:3}. We define $v_0=v/2$, and $v=v_1+v_2$ and $|v|^2=|v_1|^2+|v_2|^2$ imply that
\begin{align*}
v=v_0+|v_0|k,\quad v_1=v_0+|v_0|\sigma,\quad v_2=v_0-|v_0|\sigma,\quad \sigma\in S^{d-1}.   
\end{align*}
We define $k=\frac{v}{|v|}\in S^{d-1}$ and the angle $\theta=\arccos \frac{v\cdot\sigma}{|v|}$. By Taylor expansion, we have
\begin{align*}
    &\overline\nabla_3 \phi=\phi_1-\phi+\phi_2\\
    =&{}\phi_2+\int_0^1\frac{d}{dt}\phi\big(x,v_0+|v_0|((1-t)k+t\sigma)\big)\dd t\\
     =&{}\phi_2+|v_0|(\sigma-k)\cdot\int_0^1\nabla_v\phi\big(x,v_0+|v_0|((1-t)k+t\sigma)\big)\dd t.
\end{align*}
As $\theta\to0$, we have $\sigma\to k$ and $\phi_2\to \phi(x,0)$. Then we have
\begin{align*}
    \overline\nabla_3 \phi=\phi(x,0)+O(\theta|v_0|\|\nabla_v\phi\|_{L^\infty}).
\end{align*}
In particular, compared to Lemma \ref{lem:limit}, the leading term of 
\begin{align*}
\int_{S^{d-2}_{k^\perp}} ff_1f_2\overline\nabla_3  f^{-1} \overline\nabla_3  \phi\dd p   
\end{align*}
is $-|S^{d-2}|\phi(x,0)f^2(x,v)\sim 1$, which is different from the four wave equation case in which $\int_{S^{d-2}_{k^\perp}} ff_*f'f_*'\overline\nabla  f^{-1} \overline\nabla  \phi\dd p\sim \theta^2$.
\end{remark}

\subsection{GENERIC formulation for small-angle limiting equation}
\label{sec-3:generic}
We now show that the limiting equation \eqref{Landau-2} can also be written as a GENERIC system. For any $\phi=\phi(x,v)$, we define
$\widetilde\nabla$ for any $\phi=\phi(x,v)$ by
\begin{align*}
\widetilde\nabla\phi=\Pi_{(v-v_*)^\perp}\big(\nabla_v\phi-\nabla_{v_*}\phi_*\big).
\end{align*}
Let $G=G(x,v,v_*):\R^{3d}\to \R^d$. We have the following integration by parts formula 
\begin{align*}
    \int_{\R^{3d}}G\cdot \widetilde \nabla \phi\dd v_*\dd v\dd x=-\int_{\R^{2d}}\widetilde\nabla\cdot G \phi\dd v\dd x,
\end{align*}
and $\widetilde\nabla\cdot G$ is given by 
\begin{equation}
\label{tilde-2}
\widetilde\nabla\cdot G(x,v)=\nabla_v\cdot\int_{\R^d}\Pi_{(v-v_*)^\perp}\big(G(x,v,v_*)-G(x,v_*,v)\big)\dd v_*.  
\end{equation}

Then $Q_L(f)$ can be written as 
\begin{align*}
Q_L(f)=-2\pi\widetilde\nabla\cdot \big(B_0(ff_*)^2\widetilde\nabla f^{-1}\big).
\end{align*}

We have the following weak formulation for the limit equation \eqref{Landau-2}
\begin{align*}
   &\int_{\R^{2d}}\phi_0f_0\dd v\dd x- \int_0^T\int_{\R^{2d}}(\d_t\phi +v\cdot \nabla_x\phi)f\dd v\dd x\dd t\\
&=2\pi\int_0^T\int_{\R^{2d}}B_0(ff_*)^2\widetilde\nabla\phi\cdot \widetilde\nabla f^{-1}\dd v\dd x\dd t\quad\forall \phi\in C^\infty_c([0,T)\times\R^{2d}).
\end{align*}
Since $\widetilde\nabla(1,v,|v|^2)=0$, at least formally, we have mass, momentum and energy conservation laws
\begin{align*}
    \int_{\R^{2d}}(1,v,|v|^2)f_t\dd v\dd x=\int_{\R^{2d}}(1,v,|v|^2)f_0\dd v\dd x\quad \text{for all $t\in[0,T]$}.
\end{align*}
We recall the entropy \eqref{H}
\begin{align*}
\cH(f)=\int_{\R^{2d}}\log f\dd v\dd x.
\end{align*}
We have the following $\cH$-theorem
\begin{align*}
    \frac{d}{dt}\cH(f)=\mathcal{D}_L(f),
\end{align*}
where the entropy dissipation is given by
\begin{align*}
    \mathcal{D}_L(f)=2\pi\int_{\R^{4d}}B_0(ff_*)^2|\widetilde \nabla f^{-1}|^2\dd v_*\dd v\dd x\ge0.
\end{align*}
The limiting equation \eqref{Landau-2} can be cast as a GENERIC system with the following building block
\begin{equation*}
    \label{block:limit}
\begin{aligned}
    \aaE(f)=\int_{\R^{2d}} \frac{|v|^2}{2} f\dd v\dd x\quad&\text{and}\quad \aS(f)=\cH(f),\\
\aM(f)g=-2\pi\widetilde{\nabla}\cdot\Big(B_0(ff_*)^2\widetilde{\nabla} g\Big)\quad&\text{and}\quad\aL(f)g=-\nabla\cdot(f \aJ\nabla g),
\end{aligned}
\end{equation*}
where $\nabla=(\nabla_x,\nabla_v)^T$. Following Section \ref{sub-sec:3} and \ref{sub-sec:4}, we can check directly that $\{\aL, \aM, \aaE, \aS\}$ is indeed a GENERIC building block for the limit equation \eqref{Landau-2}. 

We now discuss the similarity and difference between the GENERIC structure of the WKE and the limiting systems with that of the Boltzmann and Landau equations. Recall that in Remark \ref{rmk:boltzmann}, we demonstrate that the Boltzmann equation can be cast as a GENERIC system with the building blocks 
\begin{align*}
 \aaE(f)=\int_{\R^{2d}} \frac{|v|^2}{2} f\dd v\dd x\quad&\text{and}\quad \aS(f)=-\cH_B(f),\\
\aM(f)g=\frac14\overline{\nabla}\cdot\Big(|V|^2\Lambda(f)\overline{\nabla} g\Big)\quad&\text{and}\quad\aL(f)g=-\nabla\cdot(f \aJ\nabla g),
\end{align*}
where $\cH_B(f)=\int f\log f$ denotes the Boltzmann entropy.
The limit Landau equation also has a GENERIC structure of the same building block with  $\aM$ replaced by
\begin{equation*}
\label{block-landau}
\aM(f)g=\frac12\widetilde{\nabla}\cdot\Big(B_0ff_*\widetilde{\nabla}g\Big),
\end{equation*}
where $\widetilde\nabla$ is defined as in \eqref{tilde-2}.
A variational characterisation for the GENERIC structure of spatial-homogeneous and delocalised Landau has been studied in \cite{carrillo2024landau,DuongHe2025}.
In particular, \textcite{carrillo2022boltzmann} showed the grazing limit of spatial-homogeneous Boltzmann equations by using of the GENERIC (gradient flow) of the homogeneous Boltzmann and Landau equations.

%{  Similar to the discussion of the Boltzmann equation \eqref{Boltzmann} and the four-wave kinetic equation \eqref{4WKE} in Remark \ref{rmk:heat}, the grazing limit kinetic Landau and wave Landau equations
%\begin{equation}
%\label{G-3}
%\begin{aligned}
%&\d_t f+v\cdot \nabla_xf=\frac12\widetilde\nabla\cdot\big(B_0^2 ff_*\widetilde\nabla\log f\big),\\
%&\d_t f+v\cdot \nabla_xf=-2\pi\widetilde\nabla\cdot\big(B_0^2 (ff_*)^2\widetilde\nabla f^{-1}\big)
%\end{aligned}
%\end{equation}
%can also be viewed as systems describing diffusion and heat conduction phenomena associated with an underlying dynamical system, induced by the Boltzmann and heat conduction entropies
%\begin{align*}
%    \cH_B(f)=-\int f\log f\quad\text{and}\quad\cH(f)=\int \log f, 
%\end{align*}
%together with the linear and quadratic mobilities $ff_*$ and $(ff_*)^2$.
%}
%{\color{blue} should keep the green below? }
{

We compare the GENERIC structures of the wave Landau equation with the WKE, Boltzmann, and Landau equations from two complementary perspectives
\begin{align}
&(\d_t+v\cdot \nabla_x)f=-\pi\overline\nabla\cdot\big(|v-v_*|^\gamma ff_*f'f_*'\overline\nabla  f^{-1}\big)\tag{WKE},\\
&(\d_t+v\cdot \nabla_x)f=-2\pi\widetilde\nabla\cdot\big(|v-v_*|^{\gamma+2}  (ff_*)^2\widetilde\nabla f^{-1}\big)\tag{wave Landau},\\
&(\d_t+v\cdot \nabla_x)f=\frac14\overline\nabla\cdot\big(|v-v_*|^\gamma\Lambda(f)\overline\nabla\log f\big)\tag{Boltzmann},\\
 &(\d_t+v\cdot \nabla_x)f=\frac12\widetilde\nabla\cdot\big(|v-v_*|^{\gamma+2} ff_*\widetilde\nabla\log f\big).\tag{Landau}
\end{align}

On the one hand, we compare the equations within the two pairs: the WKE with the wave Landau equation, and the Boltzmann equation with the Landau equation. In each pair, the two equations share the same GENERIC building blocks and differ only in the operator $\aM$, corresponding respectively to the gradients $\overline\nabla$ and $\widetilde\nabla$.

On the other hand, we compare the equations across these two pairs: the wave Landau equation with the Landau equation, and the WKE with the Boltzmann equation. In each pair, the two equations are associated with the heat-conduction entropy and the Boltzmann entropy 
\begin{align*}
    \cH(f)=\int \log f\quad\text{and}\quad \cH_B(f)=-\int f\log f
\end{align*}
and the linear and quadratic mobilities $ff_*$ and $(ff_*)^2$, respectively. This is analogous to the heat equation discussed in Remark \ref{rmk:heat}.
}

\begin{remark}
As a consequence of the GENERIC formulations for the 3-wave, 4-wave kinetic equations and the small-angle limiting equation,  by simply letting $\aL=0$ and $\aaE=0$, it follows that the corresponding spatially homogeneous models are gradient flows.
\end{remark}

\section{Conclusion and Outlook}
\label{sec: conclusion}
We have formulated the  three-wave and four-wave kinetic equations into the GENERIC framework and formally derived the small-angle limit for the four-wave equation. In addition, we have also shown the GENERIC structure of the limiting system. These establish new analogy to the Boltzmann and Landau equations in the classical kinetic theory.

By casting the WKE into the GENERIC framework, our analysis sheds some new light into these equations.
\begin{itemize}
\item \textit{(Thermodynamic consistency)} Thermodynamic consistency built-in
GENERIC guarantees the conservation of the energy and the entropy production structure. So the $H$-theorem of the WKE is encoded structurally, not just as an isolated property of the collision integral.
\item \textit{(Invariant and constrains)} The degeneracy conditions in GENERIC help to identify invariants and constrain dynamics naturally:
\begin{itemize}
    \item The degeneracy $\aM\dd\aaE=0$ corresponds to energy conservation by collisions.
    \item The degeneracy $\aL\dd\aS=0$ corresponds to entropy invariance by reversible flow.
\end{itemize}
This offers a systematic and compact/structural way to identify invariants (wave action, energy, momentum) without re-deriving them from scratch or guessing.
\item \textit{(Maximum entropy principle)} As shown in \cite{mielke2011formulation}, due to the degeneracy conditions, equilibria of a GENERIC equation can be obtained by the maximum entropy principle, that is they are maximizers of the entropy functional under the constrain of energy conservation. Applying this to the WKE, its stationary solutions solve (assuming isotropy)
\[
0=\dd(\aS(f)-\beta \aaE (f))-\mu=f^{-1}-\beta \omega (v)-\mu,
\]
which yields Rayleigh-Jeans equilibria:
\[
f(v)=\frac{1}{\mu+\beta \omega(v)},
\]
where the parameter $\mu$ to take into account the conservation of mass ($\mu=0$ for the 3-wave equation) and $\beta$ denotes the temperature of the system.  Note that the maximum entropy principle and stationary solutions have been discussed in the literature, see e.g. \cite{zakharov2012kolmogorov,saint2022wave,EGLM25,Men24,CDHG24}; however, the GENERIC  {formulation provides further thermodynamical and geometrical interpretation} for this statement, see \cite{mielke2011formulation}.
    \item \textit{Variational formulation}. The paper \cite{DPZ13} introduces a general variational formulation for the GENERIC system \eqref{eq:generic}
    that is: a curve $\az$ is a (variational) solution to \eqref{eq:generic} if and only if
    \begin{equation}
     \label{variatioanal formulation}   
    \aS(\az_0)-\aS(\az_T) +\frac12 \int_0^T \|\partial_t \mathsf z -\aL\dd\aaE\|_{\aM^{-1}}^2\dd t { + }\frac12 \int_0^T\|\dd\aS\|_{\aM}^2 \dd t\leq 0,
    \end{equation}
    { where $\|\cdot\|_{\aM}$ denotes the $\aM$-weighted $L^2$-(pseudo) norm,} and $\|\cdot\|_{\aM^{-1}}$ denotes the (pseudo) norm induced by $\aM$:
    \begin{equation*}
     \|y\|_{\aM^{-1}}^2:=\sup_{x}\Big\{\langle x,y\rangle-\frac{1}{2}\|x\|_{\aM}^2\Big\}.   
    \end{equation*}
This variational formulation is an extension of the celebrated De Giorgi’s energy-dissipation principle for gradient flows \cite{ambrosio2005gradient}. It has been employed to proving well-posedness, constructing structure-preserving numerical methods and analysing multiscale behaviour for a number of systems including Vlasov-Fokker-Planck equation, finite-dimensional damped Hamiltonian systems, the fuzzy Boltzmann and Landau equations \cite{jungel2021minimizing,EH25,DuongHe2025}. By applying \eqref{variatioanal formulation} and using the building block constructed in Sections \ref{sub-sec:3}-\ref{sub-sec:4}, one can obtain a variational formulation for the 3-wave and 4-wave kinetic equations.
    \end{itemize}
Our analysis also opens up a number of interesting questions:
\begin{itemize}
    \item Can one derive the GENERIC structure of the WKE from the microscopic nonlinear wave equations?
    \item Can one establish rigorously the variational formulation \eqref{variatioanal formulation} for the WKE and connect it to the rate functional from the large-deviation principle of the WKE obtained in \cite{guioth2022path}?
    \item Can one exploit the variational structure to study well-posedness, structure-preserving numerical schemes, long-time behaviour and multiscale analysis for the WKE?
\end{itemize}
We hope that the new avenues of research will lead to deeper understanding of the WKE and wave turbulence theory.

\printbibliography

\end{document}